\documentclass[11pt, twoside]{article}
\usepackage{amsfonts,amssymb,amsmath,amsthm}
\usepackage{graphicx}
\usepackage[all]{xy}
\usepackage{multirow} 
\usepackage{psfrag,xmpmulti,amscd,color,pstricks, import}
\usepackage{caption}

 \divide\oddsidemargin by 2
\newtheorem{thm}{Theorem}[section]
\newtheorem{cor}[thm]{Corollary}
\newtheorem{lem}[thm]{Lemma}
\newtheorem{prop}[thm]{Proposition}

\newtheorem{lemma}[thm]{Lemma}

\theoremstyle{definition}
\newtheorem{defn}[thm]{Definition}
\newtheorem{rem}[thm]{Remark}
\newtheorem*{rem*}{Remark}

\numberwithin{equation}{section}

\definecolor{OrangeRed}{cmyk}{0,0.6,1,0}            
\definecolor{DarkBlue}{cmyk}{1,1,0,0.20}
\definecolor{DarkGreen}{cmyk}{1,0,0.6,0.2}
\definecolor{myblue}{rgb}{0.66,0.78,1.00}
\definecolor{Violet}{cmyk}{0.79,0.88,0,0}
\definecolor{Lavender}{cmyk}{0,0.48,0,0}

\renewcommand{\Re}{\operatorname{Re}}
\newcommand{\im}{\operatorname{Im}}
\newcommand{\re}{\operatorname{Re}}

\newcommand{\dist}{\operatorname{dist}}

\renewcommand{\AA}{{\cal A}}

\newcommand{\MM}{{\cal M}}

\newcommand{\C}{{\mathbb C}}

\newcommand{\D}{{\mathbb D}}
\renewcommand{\H}{\mathbb{H}}

\newcommand{\N}{{\mathbb N}}

\renewcommand{\P}{{\mathbb P}}
\newcommand{\R}{{\mathbb R}}

\newcommand{\Z}{{\mathbb Z}}

\newcommand{\ra}{\rightarrow}
\newcommand{\ov}{\overline}

\renewcommand{\epsilon}{\varepsilon}
\renewcommand{\phi}{\varphi}

\title{Automorphisms of $\C^2$ with cycles of escaping Fatou components with hyperbolic limit sets }
 
\vspace{5cm}
\author{Veronica Beltrami\footnote{This work was partially supported by the Indam groups GNAMPA and GNSAGA and by PRIN 2022 Real and Complex Manifolds: Geometry and Holomorphic Dynamics}}

\begin{document}

\maketitle  
\begin{abstract}
We study the stable dynamics of non-polynomial automorphisms of $\C^2$ of the form $F(z,w)=(e^{-z^m}+ \delta e^{\frac{2 \pi}{m}i}\, w\,,\,z)$, with $m\ge 2$ a natural number and $\R\ni\delta>2$.

If $m$ is even, there are $\frac{m}{2}$ cycles of escaping Fatou components, all of period $2m$. If $m$ is odd there are $\frac{m-1}{2}$ cycles of escaping Fatou components of period $2m$ and just one cycle of escaping Fatou components of period $m$.

These maps have two distinct limit functions on each cycle, both of which have generic rank 1. Each Fatou component in each cycle has two disjoint and hyperbolic limit sets on the line at infinity, except for the Fatou components that belong to the unique cycle of period $m$: the latter in fact have the same hyperbolic limit set on the line at infinity.

\end{abstract}

\section{Introduction}
We consider the evolution of $\C^2$ under the iteration of non-polynomial automorphisms of $\C^2$, which are non-polynomial holomorphic maps $F:\C^2 \longrightarrow \C^2$ injective and surjective. We denote the $n$-th iteration of $F$ by $F^n$, i.e. $F$ composed with itself $n$ times.

Following \cite{henon1}, we call a family of holomorphic functions on a domain $\Omega\subset \C^2$ to $\C^2$  \textit{normal} if every sequence has a subsequence which converges uniformly on compact subsets to a holomorphic function from $\Omega$ to $\P^2$, where $\P^2$ is the compactification of $\C^2$ with the line at infinity $\ell_\infty$. The \textit{Fatou set} of $F$ is the set of points of $\C^2$ that have a neighborhood $U$ such that $\{F^n_{|_U} \}_{n\in \N}$ forms a normal family and a \textit{Fatou component} is a connected component of the Fatou set. 

Given a Fatou component $\Omega$ for $F$, we define a \textit{limit function} for $\Omega$ as a holomorphic function $h:\Omega \longrightarrow \P^2$ such that there exists a subsequence $n_j$ such that $F^{n_j}\rightarrow h$ uniformly on compact subsets of $\Omega$. The image of $\Omega$ under $h$ is called \textit{limit set} of $\Omega$ and we denote it by $h(\Omega)$. 
Furthermore we define the \textit{rank} of a limit function $h$ as the maximal rank of its differential.

A known result (see Lemma 2.4 of \cite{henon1}) concerning limit sets asserts that if $h:\Omega\longrightarrow \P^2$ is a limit function for $\Omega$ and $h(\Omega)\cap \ell_\infty \ne \emptyset$ than $h(\Omega) \subset \ell_\infty$.  

In this paper we consider \textit{escaping} Fatou components: a Fatou component $\Omega$ is called escaping if $h(\Omega)\subset \ell_\infty$, that is $\Omega$ has points whose orbits converge to $\ell_\infty$. In a certain sense they can be seen as the analogos of Baker domains in one-dimensional transcendental dynamics.
 

We are interested in a special subclass of non-polynomial automorphisms of $\C^2$, the subclass of transcendental Hénon maps, introduced for the first time in \cite{Dujardin04}. General proprerties of transcendental Hénon maps were studied in \cite{henon1}, \cite{henon2} and \cite{henon3}.
A Transcendental Hénon map has the form
\begin{equation}\label{eq:trans hen}
    F(z,w) = (f(z)+aw, z) 
\end{equation}
where $f: \C \longrightarrow \C$ is an entire transcendental function and $a\ne 0$ is a complex constant (note that $a$ is the modulus of the determinant of the Jacobian of $F$). 

Transcendental Hénon maps provide a natural extension of the well-studied class of polynomial Hénon maps, where $f$ in (\ref{eq:trans hen}) is a polynomial from $\C$ to $\C$ of degree $d\ge 2$.
\\

In particular in this paper we are interested in analyzing escaping Fatou components for transcendental Hénon maps with rank 1 limit functions, and the reasons for this are explained below.

For polynomial Hénon maps unbounded forward orbits are in the Fatou set and converge to the point $[1:0:0]\in \ell_\infty$ \cite{BS91}, hence there is always only one escaping Fatou component, which is an attracting basin of $[1:0:0]$, so the matter of existence and properties of escaping Fatou components is essentially
settled. But to date, there is no classification for escaping Fatou components for
non-polynomial automorphisms of $\C^2$.

On the other hand, regarding rank 1 limit functions, for polynomial Hénon maps it is not even known whether rank 1 limit functions can exist, more precisely in \cite{LP14} the existence of rank 1 limit functions is excluded if the Jacobian is small enough.  There are very few examples of non-polynomial automorphisms of $\C^2$ with limit functions of rank 1 (\cite{JL04}, and \cite{BTBP}), all of which have non-constant Jacobian and non-escaping Fatou components. Moreover with regard to transcendental Henon maps, there are only two examples of rank 1 limit function (with escaping Fatou components), in \cite{BSZ} the map $F(z,w)=(e^{-z}+2w,z)$ with one invariant escaping Fatou component on which there are two limit functions both of rank 1; and in \cite{BBS} the map $F(z,w)=(e^{-z^2}-\delta w,z)$, with $\delta \in \R$, $\delta>2$, has one cycle of escaping Fatou components on which there are two limit functions both of rank 1.

In this paper we consider transcendental Hénon maps of the form:

\begin{equation}\label{F}
F(z,w):=\left(e^{-z^m} +\delta e^{\frac{2 \pi}{m}i} w,z\right)
\end{equation}

that is setting $f(z)=e^{-z^m}$ and $a= \delta e^{\frac{2 \pi}{m}i} $, with $m\in \N$, $m\geq 2$ and $\delta \in \R$, $\delta>2$. Notice that the example in \cite{BBS} belongs to this class if we set $m=2$.

The main theorem is the following.
\begin{thm}[Main Theorem]\label{Main Theorem}
Let $F$ be defined as in (\ref{F}), than
  \begin{itemize}
 \item There are $m^2$ distinct Fatou components that exhibit cyclic behavior, more precisely
      \begin{itemize}
          \item If $m$ is even there are $\frac{m}{2}$ cycles of escaping Fatou components of period $2m$. \\
          \item If $m$ is odd there are $\frac{m-1}{2}$ cycles of escaping Fatou components of period $2m$ and only one cycle of escaping Fatou components of period $m$.
      \end{itemize}
 
  \item Each cycle has exactly two distinct limit function $h_1$, $h_2$, both of which have generic rank 1.\\
  \item Each Fatou component in each cycle has two disjoint and hyperbolic limit sets, with the exception of the Fatou components belonging to the only cycle of period $m$ (the one occurring when $m$ is odd), which have the same hyperbolic limit set.\\
  
  \item Denote the union of the $m^2$ components with $\Omega$, than $F$ is conjugate to the linear map $L(z,w)=(\delta e^{\frac{2\pi}{m}i} w, z)$ on $\Omega$.\\
  
  \item Each Fatou component in each cycle is biholomorphic to $\H \times \H$.
\end{itemize}
\end{thm}

The points of greatest interest are that the limit functions have rank 1, that the limit sets are hyperbolic, that we get cycles of escaping Fatou components where the dynamics vary depending on whether $m$ is even or odd.

\section{Cyclic behaviour}
As anticipated in the introduction, we consider 
$$
F(z,w):=(f(z)+ a w,z) \,\,\,\,\,\,\,\,\,\,\,\text{ with $f(z)=e^{-z^m}$ and $a= \delta e^{\frac{2 \pi}{m}i} $ },
$$ 
with $m\in \N$, $m\geq 2$ and $\delta\in \R$, $\delta >2$.

Define the following $m$ open subsets of $\C$:
$$
\mathcal{S}_k:=\left\{ z\in \C \,:\,\left | \im \Big(z \,e^{\frac{2(m-k)}{m}\pi i}\Big)\right|  < \tan\Big(\frac{\pi}{2m}\Big) \re \Big(z\, e^{\frac{2(m-k)}{m}\pi i}\Big)  \right\},
$$
with $k\in \Z_m$. And let 
$$
\mathcal{S}=\bigcup_k \mathcal{S}_k,\,\text{ with }k\in\Z_m .
$$
Observe the following simple lemma asserting that $f$ is bounded on $\mathcal{S}$.
\begin{lem}\label{lem:control of f}
Let $z\in \mathcal{S}$, then
\begin{equation}
|f(z)|=|e^{-z^m}|< 1.
\end{equation} 
\end{lem}
\begin{proof}
If $z\in \mathcal{S}_k$, then $|\arg(z)|< \frac{\pi}{2m}$ and hence $\re z^m >0$, from which we have $|e^{-z^m}|=e^{-\re z^m} < 1$.
\end{proof}

Consider also the following $m^2$ open subsets of $\C^2$
$$
S_{k_1k_2}:=\mathcal{S}_{k_1}\times \mathcal{S}_{k_2}
$$
with $k_1,k_2\in\Z_m$, and let
$$
S:=\bigcup_{k_1,k_2} S_{k_1k_2},\,\text{ with }k_1,k_2\in\Z_m.
$$
Sometimes we will use $S_{ab}$ instead of $S_{k_1 k_2}$ to simplify notation, with $a,b \in \Z_m$.

For $P=(z_0,w_0)\in \C^2$ and $n\in\N$ we define $(z_n,w_n):=F^n(P)$ the $n$-th iterate of the point $P$ under the action of $F$.
Using the expression of $F$, we can compute explicitly the iterates $F^{2n}$ and $F^{2n+1}$:
\begin{equation}\label{eq:iteratesp}
F^{2n}(z_0,w_0)= \left(a^nz_0+a^n\sum_{j=1}^na^{-j}f(z_{2j-1})\,,\,a^nw_0+a^n\sum_{j=1}^na^{-j}f(z_{2j-2})\right)    
\end{equation}
\begin{equation}\label{eq:iteratesd}
F^{2n+1}(z_0,w_0)= \left(a^{n+1}w_0+a^{n+1}\sum_{j=1}^{n+1}a^{-j}f(z_{2j-2})\,,\,a^n z_0+a^n\sum_{j=1}^n a^{-j}f(z_{2j-1})\right) .
\end{equation}

Define 
\begin{align}\label{eq:iterates1}
\Delta_1(z_0,w_0)&:=\sum_{j=1}^\infty a^{-j}f(z_{2j-1})\\
\label{eq:iterates2}\Delta_2(z_0,w_0)&:=\sum_{j=1}^\infty a^{-j}f(z_{2j-2})\\
\label{eq:Delta} \Delta &:= \sum_{j=1}^\infty |a|^{-j}= \sum_{j=1}^\infty \delta^{-j}. 
\end{align}
Notice that $\Delta=\frac{\delta}{\delta-1}-1$, and since $\delta>2$, $\Delta<1$. Moreover, using Lemma~\ref{lem:control of f}, we have the following.
\begin{rem}\label{rem: bounds on Delta in S}
Let $P=(z_0,w_0) \in S$ such that $F^n(z_0,w_0)\in S$ for all $n\in \N$, then
$$|\Delta_1(z_0,w_0)|, \,|\Delta_2(z_0,w_0)|< \Delta < 1 .$$
\end{rem} 

For the rest of this section, we shall consider $P$ to be a point in $S$ such that $F^n(P)\in S$ for all $n\in \N$.

For such $P=(z_0,w_0)$ we can also deduce the following formal limits:
\begin{equation}\label{eq:limit h1}
h_1(z_0,w_0):=\lim_{n\rightarrow \infty}\frac{z_{2n}}{w_{2n}}=\frac{z_0+\Delta_1(z_0,w_0)}{w_0+\Delta_2(z_0,w_0)}
\end{equation}
\begin{equation}\label{eq:limit h2}
h_2(z_0,w_0):= \lim_{n\rightarrow \infty}\frac{z_{2n+1}}{w_{2n+1}}=\frac{a w_0+a \Delta_2(z_0,w_0)}{z_0+\Delta_1(z_0,w_0)}=\frac{a}{h_1}
\end{equation}

and we will later show in Proposition~\ref{prop:h1 h2} that $h_1 \ne h_2$.

Notice that such a $P$ exists, indeed each $S_{ab}$ contains the set 
\begin{equation}\label{eq:Aab}A_{ab}:=\left\{ \re \left( ze^{\frac{2(m-a)}{m}\pi i} \right),\re\left( w e^{\frac{2(m-b)}{m}\pi i}\right)> M\,,\,\im \left( ze^{\frac{2(m-a)}{m}\pi i} \right),\im \left( we^{\frac{2(m-a)}{m}\pi i} \right)=0 \right\}
\end{equation}
for $M$ sufficiently large, and
$$A:=\bigcup_{a,b\in\Z_m} A_{ab} \subset S$$
is forward invariant under $F$, so for example each $P\in A$ satisfies the requirement.

\begin{lem}[Forever in $S$ implies convergence]\label{lem:forever in S1}
Let $P=(z_0,w_0)\in S$ such that $F^n(P)\in S$ for all $n \in \N$, then 
\begin{align*}
F^{2n}(z_0,w_0)&\ra h_1(z_0,w_0)\\
F^{2n+1}(z_0,w_0)&\ra h_2(z_0,w_0)\\
\end{align*}
\end{lem}
\begin{proof}
Since $F^n(P)\in S$ for all $n\in\N$ we have that  $z_n\in \mathcal{S}$ for all $n\in\N$ and hence by Lemma~\ref{lem:control of f} $|f(z_n)|< 1$ for all $n\in\N$.  This implies that $|\Delta_1(z_0,w_0)|$, $ |\Delta_2(z_0,w_0)|<\Delta<1$, which implies convergence of the even and odd iterates of $F$.
\end{proof}

With the following proposition we show that there is a cyclic behaviour.
\begin{prop}\label{prop:cicli}
Let $P=(z_0,w_0)\in S_{a b}$, such that $F^n(P)\in S$ for all $n\in \N$. If $\re z_0, \re w_0$ are sufficiently large, then $F(P)\in S_{(b+1)a}$.
\end{prop}

\begin{proof}
We prove the claim for $P=(z_0,w_0)\in S_{00}$, the other cases are analogous.

By hypothesis we know that $F(P)=(z_1,w_1)\in S_{ab}$ for some $a,b \in \Z_m$.
Since $P\in S_{00}$, we have 
$$|\im z_0|< \tan\left(\frac{\pi}{2m}\right) \re z_0$$ and $$|\im w_0|<\tan\left(\frac{\pi}{2m}\right) \re w_0,$$
moreover $w_1=z_0$, so $|\im w_1|<\tan\left(\frac{\pi}{2m}\right) \re w_1$, and so $b=0$.

Recall that $z_1=e^{-z_0^m}+\delta e^{\frac{2 \pi}{m}i}w_0$, and notice that $\delta e^{\frac{2 \pi}{m}i}w_0 \in \mathcal{S}_1$. Furthermore $|e^{-z_0^m}|<1$, then $z_1\,e^{\frac{2(m-1)}{m}\pi i}$ belongs to the $1$-neighborhood of $\mathcal{S}_0$.
Choosing $\re w_0$ sufficiently large, this $1$-neighborhood of $\mathcal{S}_0$ intersects $\mathcal{S}$ only in $\mathcal{S}_0$ and so $a=1$.
\end{proof}

To better understand the cycling behavior of the sectors $S_{ab}$, with $a,b \in \Z_m$, 
 let us consider the following application: 

$$\gamma:\Z_m\times \Z_m \longrightarrow \Z_m \times \Z_m \,\,\,\,\,\text{ defined as }\,\,\,\,\,
 \gamma(a,b):=(b+1,a)\,.$$
 
It is easy to check that
$$
\gamma^{2n}(a,b)=(a+n,b+n)
$$ and
$$\gamma^{2n+1}(a,b)=(b+n+1,a+n)\,.$$
In order to have cycles we need to set the following equations:
$$\gamma^{2n}(a,b)=(a,b) $$
$$\gamma^{2n+1}(a,b)=(a,b) $$
and we obtain respectively $2n=2m$ and $2n+1=m$.
So we can have cycles of period $2m$ or of period $m$. Since we have $m^2$ sectors, to understand how many and which cycles have a period $2m$ or $m$, we first need to solve the equation
$$m^2=A2m+Bm$$
that is $m=2A+B$, with $A,B \in \N$.

Additionally, notice that for each cycle, we can take $(0,b)$ as a representative. Therefore, let us see after how many iterations, in each cycle, we obtain $(0, \Tilde{b})$. We again consider  
$$\gamma^{2n}(0,b)=(0, \Tilde{b})$$
from which we obtain, after $2m$ iterations, $\Tilde{b}=b$, and
$$\gamma^{2n+1}(0,b)=(0,\Tilde{b})$$
from which we have, after $2m-(2b+1)$ iterations, $\Tilde{b}=m-(b+1)$  or, after $m$ iterations, $\Tilde{b}=b=\frac{m-1}{2}$. This means that we can have at most one cycle of period $m$, the one represented by $(0,\frac{m-1}{2})$.

If $m$ is even we can not have the cycle of period $m$, since $\frac{m-1}{2}\notin \N$, so in this case we get $A=\frac{m}{2}$ and $B=0$, that is we have $\frac{m}{2}$ cycles of period $2m$ and zero of period $m$. Moreover in each cycle there is $(0,b)$ and after $2m-(2b+1)$ iterations also $(0,m-(b+1))$.

If $m$ is odd, we have $B=1$ and $A=\frac{m-1}{2}$, that is there are $\frac{m-1}{2}$ cycles of period $2m$
and one cycle of period $m$, which we refer to as the \textit{short cycle}, and it is the one represented by $(0, \frac{m-1}{2})$. In each cycle of period $2m$, as in the even case, there is $(0,b)$ and after $2m-(2b+1))$ iterations also $(0,m-(b+1))$.

To better understand these cycles, let us consider two examples: $m = 5 $ and $m = 6$.

$$\begin{displaystyle}
m=5\,\,\,\,\,\,\,\,\,\,\,\,\,\,\,\,\,\,
  \begin{cases}
\textbf{00} &\textbf{ iteration 0}\\
10 &\text{ iteration 1}\\
11 &\text{ iteration 2}\\
21 &\text{ iteration 3}\\
22 &\text{ iteration 4}\\
32 &\text{ iteration 5}\\
33 &\text{ iteration 6}\\
43 &\text{ iteration 7}\\
44 &\text{ iteration 8}\\
\textbf{04} &\textbf{ iteration 9}\\
\textbf{00} &\textbf{ iteration 10}
\end{cases}
\,\,\,\,\,\,\,\,\,\,\,\,\,\,\,\,\,\,
\begin{cases}
\textbf{01} &\textbf{ iteration 0}\\
20 &\text{ iteration 1}\\
12 &\text{ iteration 2}\\
31 &\text{ iteration 3}\\
23 &\text{ iteration 4}\\
42 &\text{ iteration 5}\\
34 &\text{ iteration 6}\\
\textbf{03} &\textbf{ iteration 7}\\
40 &\text{ iteration 8}\\
14 &\text{ iteration 9}\\
\textbf{01} &\textbf{ iteration 10}
  \end{cases}
  \,\,\,\,\,\,\,\,\,\,\,\,\,\,\,\,\,\,
\begin{cases}
\textbf{02} &\textbf{ iteration 0}\\
30 &\text{ iteration 1}\\
13 &\text{ iteration 2}\\
41 &\text{ iteration 3}\\
24 &\text{ iteration 4}\\
\textbf{02} &\textbf{ iteration 5}\\    
  \end{cases}
\end{displaystyle}$$

$$\begin{displaystyle}
m=6\,\,\,\,\,\,\,\,\,\,\,\,\,\,\,\,\,\,
  \begin{cases}
\textbf{00} &\textbf{ iteration 0}\\
10 &\text{ iteration 1}\\
11 &\text{ iteration 2}\\
21 &\text{ iteration 3}\\
22 &\text{ iteration 4}\\
32 &\text{ iteration 5}\\
33 &\text{ iteration 6}\\
43 &\text{ iteration 7}\\
44 &\text{ iteration 8}\\
54 &\text{ iteration 9}\\
55 &\text{ iteration 10}\\
\textbf{05} &\textbf{ iteration 11}\\
\textbf{00} &\textbf{ iteration 12}
\end{cases}
\,\,\,\,\,\,\,\,\,\,\,\,\,\,\,\,\,\,
\begin{cases}
\textbf{01} &\textbf{ iteration 0}\\
20 &\text{ iteration 1}\\
12 &\text{ iteration 2}\\
31 &\text{ iteration 3}\\
23 &\text{ iteration 4}\\
42 &\text{ iteration 5}\\
34 &\text{ iteration 6}\\
53 &\text{ iteration 7}\\
45 &\text{ iteration 8}\\
\textbf{04} &\textbf{ iteration 9}\\
50 &\text{ iteration 10}\\
15 &\text{ iteration 11}\\
\textbf{01} &\textbf{ iteration 12}
  \end{cases}
  \,\,\,\,\,\,\,\,\,\,\,\,\,\,\,\,\,\,
\begin{cases}
\textbf{02} &\textbf{ iteration 0}\\
30 &\text{ iteration 1}\\
13 &\text{ iteration 2}\\
41 &\text{ iteration 3}\\
24 &\text{ iteration 4}\\
52 &\text{ iteration 5}\\
35 &\text{ iteration 6}\\
\textbf{03} &\textbf{ iteration 7}\\
40 &\text{ iteration 8}\\
14 &\text{ iteration 9}\\
51 &\text{ iteration 10}\\
25 &\text{ iteration 11}\\
\textbf{02} &\textbf{ iteration 12}
  \end{cases}
\end{displaystyle}$$

In the following proposition we analyze how the real part of $P$ increases.
\begin{prop}[Growth of the real part]\label{prop:come cresce re}
Let $P=(z_0,w_0)\in S_{00}$, such that $F^n(P)\in S$ for all $n\in \N$. Then for all $\lambda>0$, if $\re w_0,\re z_0>\frac{1+\lambda}
{\delta-1}$
$$
\re \left(z_{2n-1}\,e^{\frac{2(m-(2n-1))}{m}\pi i}\right)=\re \left(w_{2n}\,e^{\frac{2(m-(2n-1))}{m}\pi i}\right)> \re w_0+n\lambda \,,
$$
$$
\re \left(z_{2n}\,e^{\frac{2(m-(2n-1))}{m}\pi i}\right)=\re \left(w_{2n+1}\,e^{\frac{2(m-(2n-1))}{m}\pi i}\right)> \re z_0+n\lambda\, .
$$
\end{prop}
\begin{proof}
Let $P=(z_0,w_0)$ as in the hypothesis. Since $P\in S$, then by Lemma~\ref{lem:control of f}, we have
\begin{align*}
\re \left( z_1\,e^{\frac{2(m-1))}{m}\pi i}\right)&=\re \left(e^{-z_0^m}\,e^{\frac{2(m-1))}{m}\pi i} \right) + \delta \re w_0>\\
&>\delta \re w_0 -\left|\re\left( e^{-z_0^m}\,e^{\frac{2(m-1))}{m}\pi i}\right) \right|> \delta \re w_0 - 1
\end{align*}
so
\begin{equation}\label{P in S}
    \re \left( z_1\,e^{\frac{2(m-1))}{m}\pi i}\right)>\delta \re w_0 - 1
\end{equation}
 which is larger than $\re w_0+\lambda$ if $\re w_0>\frac{1+\lambda}{\delta -1}$ as required. The claim for $z_{2}$ follows because $w_1=z_0$ and the more general formula follows by induction.
\end{proof}

If we substitute $S_{00}$ with a generic $S_{ab}$, with $a,b \in \Z_m$, in the Proposition~\ref{prop:come cresce re}, we obtain the following.

$$
\re \left(z_{2n-1}\,e^{\frac{2(m-b-(2n-1))}{m}\pi i}\right)=\re \left(w_{2n}\,e^{\frac{2(m-b-(2n-1))}{m}\pi i}\right)> \re w_0+n\lambda \,,
$$
$$
\re \left(z_{2n}\,e^{\frac{2(m-a-(2n-1))}{m}\pi i}\right)=\re \left(w_{2n+1}\,e^{\frac{2(m-a-(2n-1))}{m}\pi i}\right)> \re z_0+n\lambda\, .
$$

\section{Construction of a forward invariant open set $W$}

The purpose of this section is to construct a forward invariant open set $W\subset S$, that is $F(W)\subset W$, and so Lemma \ref{lem:control of f}, Lemma \ref{lem:forever in S1}, Proposition \ref{prop:cicli} and Proposition \ref{prop:come cresce re} hold on $W$.
With this in mind let us  introduce the following $m$ subsets of $\C$:
$$
\left(\mathcal{W}_{\sigma,R}\right)_k:=\left\{ z \in \C\,:\,\left | \im \Big(z \,e^{\frac{2(m-k)}{m}\pi i}\Big)\right|  < \sigma \re \Big(z\, e^{\frac{2(m-k)}{m}\pi i}\Big)  \right\} \subset \mathcal{S}_k ,
$$
with $0<\sigma< \tan(\pi/2m)$, $\re\Big(z\, e^{\frac{2(m-k)}{m}\pi i}\Big)>R$ and $k\in \Z_m$.

Define also the following $m^2$ subsets of $\C^2$:
$$
\left( W_{\sigma,R_1,R_2}\right)_{k_1 k_2}:=\left(\mathcal{W}_{\sigma,R}\right)_{k_1} \times \left(\mathcal{W}_{\sigma,R}\right)_{k_2}\subset S_{k_1k_2}
$$
with $k_1, k_2\in Z_m$ and let
$$ 
W_{\sigma,R_1,R_2}:=\bigcup_{k_1,k_2}\left( W_{\sigma,R_1,R_2}\right)_{k_1k_2} .
$$

\begin{prop}[Amplitude]\label{prop:im/re}
    Let $(z_0,w_0)\in \left( W_{\sigma,R_1,R_2}\right)_{00} $ and let $0<\sigma<\Tilde{\sigma}<\tan(\pi/2m)<1$. if $R_2>\frac{2}{\delta (\Tilde{\sigma}-\sigma)}$, then
    $$
     \frac{\left|\im \left(z_1 e^{\frac{2(m-1)}{m}\pi i}\right)\right|}{\re\left(z_1 e^{\frac{2(m-1)}{m}\pi i}\right)} < \Tilde{\sigma}\,.
    $$
\end{prop}

\begin{proof}
Let $(z_0,w_0)\in \left( W_{\sigma,R_1,R_2}\right)_{00} $, using Lemma~\ref{lem:control of f} and the fact that $|\im  w_0|<\sigma \re w_0$ we have
  $$
     \frac{\left|\im \left(z_1 e^{\frac{2(m-1)}{m}\pi i}\right)\right|}{\re\left(z_1 e^{\frac{2(m-1)}{m}\pi i}\right)} < \frac{\delta |\im w_0|+1}{\delta \re w_0-1}<\frac{\delta \sigma \re w_0+1}{\delta \re w_0-1}
$$  
which is less than $\Tilde{\sigma}$ if $\re w_0>\frac{1+\Tilde{\sigma}}{\delta (\Tilde{\sigma}-\sigma)}$. Since $\Tilde{\sigma}<1$, it is enough to take $\re w_0>\frac{2}{\delta (\Tilde{\sigma}-\sigma)}$ as required.

\end{proof}

If we substitute $S_{00}$ with a generic $S_{ab}$ with $a,b \in \Z_m$, in the Proposition~\ref{prop:im/re}, we obtain the following:
 $$
     \frac{\left|\im \left(z_1 e^{\frac{2(m-1-b)}{m}\pi i}\right)\right|}{\re\left(z_1 e^{\frac{2(m-1-b)}{m}\pi i}\right)} < \Tilde{\sigma}\,.
    $$

To simplify notation, we introduce 
$\mu:\N\ra\Z_m\times\Z_m$, which denotes the cyclic behaviour, defined as 
$$
\mu(n):=
\begin{cases}
0\,b &\text{ if $n=0\mod 2m$}\\
(b+1)\,0 &\text{ if $n=1\mod 2m$}\\
1\,(b+1) &\text{ if $n=2\mod 2m$}\\
(b+2)\,1 &\text{ if $n=3\mod 2m$}\\
2\,(b+2) &\text{ if $n=4\mod 2m$}\\
.\\
.\\
.\\
0\,(m-1) &\text{ if $n=2m-1\mod 2m$}\\
0\,b &\text{ if $n=2m\mod 2m$}
\end{cases}
$$
Notice that for the short cycle, we can consider the following
$$
\mu(n)=
\begin{cases}
0\,b &\text{ if $n=0\mod m$}\\
(b+1)\,0 &\text{ if $n=1\mod m$}\\
1\,(b+1) &\text{ if $n=2\mod m$}\\
(b+2)\,1 &\text{ if $n=3\mod m$}\\
2\,(b+2) &\text{ if $n=4\mod m$}\\
.\\
.\\
.\\
0\,(m-1) &\text{ if $n=m-1\mod m$}\\
0\,b &\text{ if $n=m\mod m$}
\end{cases}
$$

Let $\sigma_n:=\left(\frac{n+1}{n+2}\right) \tan(\frac{\pi}{2m})$ and $R_n:=(\frac{\delta}{2})^{\frac{n}{2}} R_0$ for $R_0$ sufficiently large depending only on $\delta$. Notice that $\sigma_n \in \left(0, \tan(\frac{\pi}{2m})\right)$ far all $n\in \N$.
Set 
\begin{equation}\label{eq:W_n}
W_n:=\left(W_{\sigma_n,R_n,R_{n-1}}\right)_{\mu(n)}\,,
\end{equation}
and define
\begin{equation}\label{eq:W}
W:=\bigcup_{n\in\N} W_n \,.
\end{equation}
Observe that $W\subset S$, it is open and consists of $m^2$ connected components. 
We define $W_{ab} $ with $a,b \in \Z_m$, the component of $W$ contained in $S_{ab}$. 

\begin{prop}[Invariance of $W$]\label{prop: invariance W}
We have that $F(W_n)\subset W_{n+1}$. In particular, $W$ is forward invariant.
\end{prop}
This also implies that if $W_{ab}$ does not belong to the short cycle it is forward invariant under $F^{2m}$, otherwise it is forward invariant under $F^{m}$.

\begin{proof}
Let $(z_0, w_0)\in W_n$, with $n=0 \mod 2m$, the other cases are analogous, and let $(z_1,w_1)$ be its image. Notice that $w_1=z_0$, so $\re w_1=\re z_0>R_n$, and 
$$\frac{|\im w_1|}{\re w_1}=\frac{|\im z_0|}{\re z_0}<\sigma_{n}<\sigma_{n+1}\,\,,$$
hence to show that $F(W_n)\subset W_{n+1}$ it is enough to prove that
\begin{enumerate}
\item $\re \left(z_1 e^{\frac{2(m-1)}{m}\pi i}\right)>R_{n+1}$
\item $\frac{\left|\im \left(z_1 e^{\frac{2(m-1)}{m}\pi i}\right)\right|}{\re \left(z_1 e^{\frac{2(m-1)}{m}\pi i}\right)}<\sigma_{n+1}$.
\end{enumerate}
Let $\lambda_n:=R_{n+1}-R_{n-1}$. Since $(z_0,w_0)\in S$, from (\ref{P in S}) we have $\re \left(z_1 e^{\frac{2(m-1)}{m}\pi i}\right)>\re w_0+\lambda_n > R_{n-1}+\lambda_n=R_{n+1}$ provided $R_{n-1}>\frac{1+\lambda_n}{\delta-1}$. Substituting the expression for $\lambda_n$ we get $R_{n+1}<\delta R_{n-1}-1$. Substituting the expression for $R_{n+1}$ and $R_{n-1}$ we get 
$$
\delta^{\frac{n+1}{2}}R_0>2^{\frac{n+1}{2}}
$$
which is satisfied because $\delta>2$, provided $R_0\ge1$. This gives $\re \left(z_1 e^{\frac{2(m-1)}{m}\pi i}\right)>R_{n+1}$.\\

We now show $\frac{\left|\im \left(z_1 e^{\frac{2(m-1)}{m}\pi i}\right)\right|}{\re \left(z_1 e^{\frac{2(m-1)}{m}\pi i}\right)}<\sigma_{n+1}$.
In view of Proposition~\ref{prop:im/re}, it is enough to check that  $R_{n-1}>\frac{2}{\delta(\sigma_{n+1}-\sigma_n)}=\frac{2(n+2)(n+3)}{\delta} \tan(\frac{\pi}{2m})$, that is 
$$R_0>\frac{2^{\frac{n+1}{2}}}{\delta^{\frac{n+1}{2}}}(n+2)(n+3) \tan\left(\frac{\pi}{2m}\right)\,\,.$$ 
Since the function on the right hand side is bounded in $n$ for any $\delta>2$ (moreover, it tends to $0$ as $n\ra\infty$), such $R_0$ exists and depends only on $\delta$. 
\end{proof}

\subsection{Fatou components and rank 1 limit functions}

In this section, we establish that $W$ is contained in the Fatou set. This is achieved by leveraging Lemma~\ref{lem:forever in S1} in conjunction with the observation that $W$ possesses the characteristics of being non-empty, open, forward invariant, and contained in $S$. Additionally, we demonstrate that the functions $h_1$ and $h_2$, as defined in (\ref{eq:limit h1}) and (\ref{eq:limit h2}) respectively, have generic rank 1, and further, that $h_1 \ne h_2$.

\begin{prop} [Existence of Fatou components]\label{prop:existence of FC}
On each $ W_{ab}$ we have that 
$$
F^{2n}\ra h_1, F^{2n+1}\ra h_2 \text{ uniformly on compact subsets of  $ W_{ab}$}.
$$
It follows that each $W_{ab}$  is contained in a  Fatou component $\Omega_{ab}$. 
\end{prop}

\begin{proof}
Points in each $W_{ab}$ never leave $S$ by Proposition~\ref{prop: invariance W}. Hence the even and odd iterates of $F$ converge according to Lemma~\ref{lem:forever in S1} on compact subsets of each  $W_{ab}$. Since each $W_{ab} $ is open and connected it is contained in a Fatou component $\Omega_{ab}$. 
\end{proof}

Notice that in Proposition~\ref{prop:existence of FC} we define $\Omega_{ab}$ to be the Fatou component containing $W_{ab}$ with $a,b \in \Z_m$. Let  $$\Omega:=\bigcup_{ab}\Omega_{ab},.$$
The following corollary is a direct consequence of Proposition~\ref{prop:existence of FC}.

\begin{cor}\label{cor:at most}
    The set $\Omega$ consists of at most $m^2$ connected components.
\end{cor}

Moreover we will see in Proposition~\ref{prop: four different FC} that the components $\Omega_{ab}$ are in fact all distinct, so $\Omega$ consists of exactly $m^2$ connected components.

We now show that $h_1,h_2$ are distinct and have generic rank 1. 
\begin{prop}\label{prop:h1 h2}
Both  $h_1$ and $h_2$ have (generic) rank $1$, and $h_1 \ne h_2$.
\end{prop}
 
\begin{proof}
 Notice that Proposition \ref{prop:come cresce re} implies that $h_i(W)$ is contained in the line at infinity and so, by Sard's Theorem, 
$h_1$ and $h_2$ have generic rank at most 1. We now show that $h_1$ and $h_2$ are non-constant, so we can conclude that they have rank 1. Suppose by contradiction that 
 $|h_1|=c$ is constant.
 
If $c\ne 0, \infty$, then one has:
$$
|z_0|-\Delta\leq \left|z_0+\sum_{j=1}^\infty a^{-j}f(z_{2j-1})\right|=c\left|w_0+\sum_{j=1}^\infty a^{-j}f(z_{2j-2})\right|\leq c|w_0|+c\Delta,
$$
hence 
$$
|z_0|\leq c|w_0|+(c+1)\Delta,
$$

contradicting the fact that $(z_0,w_0)$ could be any point in   $W$, which is unbounded in the $z$ direction for any choice of $w$.\\
If $c=0$, we have $|z_0| \le \Delta$, while if  $c=\infty$,  we have $|w_0| \le \Delta$;  in either case we have a contradiction. 

This also implies that $h_1 \ne h_2$. Indeed, $h_1 \cdot h_2$ is constant, so if we had $h_1=h_2$ we would have that $h_1^2$ is constant and hence so is $h_1$. 
\end{proof}

\subsection{Construction of an absorbing set $W_I$ for $\Omega$}

This section is dedicated to the construction of an absorbing set $W_I$ for $\Omega$ under $F$ (Proposition \ref{prop:Absorbing WI}) and to do this we use the plurisubharmonic method (for references see \cite{FornaessShortCk}, \cite{henon1}, \cite{BSZ} and \cite{BBS}). 
This fact will be used in Section~\ref{sect: limt sets} to show that the Fatou components $\Omega_{ab}$ are all distinct and to describe both their limit sets and their geometric structure.

\begin{defn}
    Let $\Omega$ be an open set, a set $A\subset \Omega$ is absorbing for $\Omega$ under a map $F$ if for any compact set $K\subset \Omega$ there exists $N>0$ such that 
    $$F^n(K)\subset A \text{ for all } n\geq N.$$
\end{defn}
 Remember that in Proposition~\ref{prop:existence of FC} we define $\Omega_{ab}$ to be the Fatou component containing $W_{ab}$ with $a,b \in \Z_m$. And $$\Omega:=\bigcup_{ab}\Omega_{ab}.$$
 
 Fix $C\ge 1$ and let 
 $$
 I=I(C):=\{ z\in \C\,:\,\re (z^m) > C^m \} \subset \mathcal{S}.
 $$
 Notice that $I$ consists of $m$ connected component, each of which is contained in one of the $\mathcal{S}_k$, so we define $I_k$ the component of $I$ contained in $\mathcal{S}_k$, with $k\in \Z_m$.
 
 Notice that if $z\in I_k$, then $\re \left (z e^{\frac{2(m-k)}{m}\pi i} \right)> C$. Define the following subset of $S$
 \begin{equation}\label{eq:W_I}
W_I=W_{I}(C):=\{(z,w) \in \C^2: F^n(z,w)\in I\times I \text{ for all $n\in\N$}\}\cap\Omega\,\,
  \end{equation}
and let
$$
  \AA_{I}=\AA_{I}(C):=\bigcup_n F^{-n} (W_{I})\,\,.$$

Our next goal is to show that $W_I$ is an absorbing set for $\Omega$ under $F$, that is $\AA_I=\Omega$.

Since $W_I\subset S$, we define $(W_I)_{ab}$ the subset of $W_I$ contained in $S_{ab}$.  
Notice that $W_I$  is forward invariant by construction and that each $(W_I)_{ab}$ contains the set $A_{ab}$ already defined in~(\ref{eq:Aab}), for $M$ sufficiently large; hence that they are all not empty. It will turn out in Corollary~\ref{cor: w_I open} that $W_I$ is also an open set.

Since  $ W_I \subset S$ and forward invariant,    Proposition~\ref{prop:cicli} holds and hence the sets $(W_I)_{ab}$ are mapping to each other: $F\left( (W_I)_{ab} \right)\subset(W_I)_{(b+1)a}$.

Moreover $(W_I)_{ab}$ is forward invariant under $F^ {2m}$, in particular if $(W_I)_{ab}$ belongs to the short cycle, it is forward invariant under $F^ {m}$. By Lemma~\ref{lem:forever in S1} we have convergence of even and odd iterates of $F$ on $W_I$.

From now on the entire section is devoted to prove the following proposition:

\begin{prop}[$W_I$ is absorbing for $\Omega$]\label{prop:Absorbing WI}
The set $W_I$  is absorbing for $\Omega$ under $F$, that is, $\AA_I=\Omega$.
\end{prop}

Define
$$\mathcal{X}:=\{ (z,w)\in \Omega \,:\,h_1(z,w)=0,\infty \}\,\,,$$
and observe that since $\mathcal X$ is an analytic set, being the union of the $0$-set and the $\infty$-set of a meromorphic function, it is locally a finite union of $1$-complex-dimensional varieties.

Let $K$ be a compact subset of $\Omega \setminus\mathcal{X}$, that is $h_1(P)\ne 0,\infty$ for all $P\in K$. We can define the quantities 
\begin{equation}
\begin{split}
M&:=\max_{K}(\max(|h_1|,|h_2|))<\infty\\
m&:=\min_{K}(\min(|h_1|,|h_2|))>0.\\
\end{split}\nonumber
\end{equation}
Note that  $M>1$ because $|h_2|=\frac{\delta}{|h_1|}$ and $\delta>2$. 
By Corollary 2.3 in \cite{BSZ}, if $0<\epsilon<m$ there exists a constant $c$ such that for every $(z_0,w_0)\in K$, 
 \begin{equation}\label{eqtn:sandwich growth}
 |z_n|\leq c(M+\epsilon)^n.
 \end{equation}
 Recall that $w_n=z_{n-1}$,hence 
 \begin{equation}\label{eqtn:sandwich growth2}
  |w_n|\leq c(M+\epsilon)^{n-1}.
 \end{equation}

Now consider the following two remarks, the first one is certainly well-known: for a proof, please refer to Appendix of \cite{BBS}. Given a set $L$, we denote its interior with $\mathring{L}$.
\begin{rem}\label{rem:estensione}
    Let $L$ be a compact set and $H$ be an analytic subset of dimension one of $\C^2$. For any compact $K$ such that $K\subset \mathring{L}$ there exists $\eta=\eta(K,L,H)$ such that for any $u$ harmonic defined in a neighborhood of $L$ and such that 
$$
u\leq\alpha<\infty \text{ on $L\setminus( \eta-$neighborhood of $H$) }
$$
we have 
$$
u\leq\alpha \text{ on $K$ .}
$$
\end{rem}

\begin{rem}\label{rem:PolyC}
Recall that $\cos(m \alpha)=T_m\left(\cos(\alpha)\right)$, with $m\in \N$, where $T_m$ are the Chebyshev polynomials of the first kind defined as
$$T_m(x)=\sum_{h=0}^{\left[\frac{m}{2}\right]} (-1)^h \binom{m}{2h} x^{m-2h} (1-x^2)^h ,$$
where $\left[\frac{m}{2}\right]$ is the integer part of $\frac{m}{2}$.
\end{rem}

The proof of Proposition~\ref{prop:Absorbing WI} relies on the following technical lemma. Recall that for a point $P=(z_0,w_0)$, we define $(z_n,w_n):=F^n(P)$.

\begin{lem}\label{lem:harmonic functions} 
Define the sequence of harmonic functions $u_n$ from  $\Omega$ to $\R$ as $u_n(z_0,w_0):=\frac{-\Re (z_n^m)}{n}$. Then 
\begin{enumerate}
\item Let $K\subset\Omega$ be a compact set, then there exists $M=M(K)$ and $N \in \mathbb N$ such that $u_n\leq \log M$ on $K$ for $n>N$.;
\item $u_n\ra -\infty$  uniformly on compact subsets of  $W$; 
\item If $P\in \Omega\setminus \AA_I$, then for all $\epsilon>0$ there exists a subsequence $n_k\longrightarrow \infty$ such that $u_{n_k}(P)\ge -\epsilon$.
 \end{enumerate}
\end{lem}
We will later show that such a $P\in \Omega\setminus \AA_I$ leads to a contradiction.

\begin{proof}
\begin{enumerate}
\item 

Let $K\subset \Omega$ compact. Let $\eta$ as in Remark~\ref{rem:estensione} applied to a slightly larger compact set $L\subset\Omega$ and to the analytic set $\mathcal{X}$. Let $U_\eta(\mathcal{X})$ be an $\eta$-neighborhood of $\mathcal{X}$.
Because of Remark~\ref{rem:estensione} it is enough to show that there exists $M$ and $N\in \N$ such that $u_n\le \log M$  for $n>N$ on the set
$$
K\setminus U_\eta(\mathcal{X})
$$
which is  a compact subset of $\Omega\setminus \mathcal{X}$. Hence it is enough to prove the claim for any compact subset $K$ of $\Omega\setminus \mathcal{X}$. \\

Fix $\epsilon \in (0,m)$  and let $c$ as in (\ref{eqtn:sandwich growth}) and (\ref{eqtn:sandwich growth2}). Suppose that there exists a subsequence $(n_j)$ and points $(z,w)=(z(j), w(j))\in K$  such that
$$-\frac{\re({z}_{n_j}^m)}{n_j}>\beta$$
for some $\beta$. We will show that $\beta\leq M$.\\ 

We have that
\begin{align*}
    |z_{n_{j+1}}|&=|e^{-z_{n_j}^m}+\delta e^{\frac{2 \pi}{m}i} w_{n_j}| \ge |e^{-z_{n_j}^m}|-\delta|w_{n_j}|\ge \\
    &\ge e^{-\re(z_{n_j}^m)}-\delta c(M+\epsilon)^{n_j-1} \ge \\
    &\ge e^{\beta n_j}-\delta c (M+\epsilon)^{n_{j}-1}\,\,\,.
    \end{align*}
Furthermore 
$$|z_{n_{j+1}}|\le c(M+\epsilon)^{n_j+1}\,\,\,.$$
Then we have 
$$e^{\beta n_j}-\delta c (M+\epsilon)^{n_{j}-1}\le c(M+\epsilon)^{n_j+1} \,\,\,,$$
that is
\begin{align*}
   e^{\beta n_j}\le \delta c (M+\epsilon)^{n_{j}-1}+ c(M+\epsilon)^{n_j+1}\,\,\,.
\end{align*}
Since $M>1$ and $\epsilon >0$, we have that $(M+\epsilon)>1$ and hence
\begin{align*}
    e^{\beta n_j}< \delta c (M+\epsilon)^{n_{j}+1}+ c(M+\epsilon)^{n_j+1}=c(\delta +1)(M+\epsilon)^{n_j+1}\,\,\,.
\end{align*}
Then
$$\beta < \frac{\log\big(c(\delta +1)\big)}{n_j}+\frac{n_j+1}{n_j}\,\, \log(M+\epsilon)$$
which implies, using $n_j\longrightarrow \infty$ and $\epsilon \longrightarrow 0$, that $\beta \le \log M$.

\item 
Let $K$ be a compact subset of $W$, since $W$ is forward invariant, $F^n(K)\subset W$ for all $n\in \N$. Moreover there exist $j \in \N$ such that $K \subset W_{j}$ defined in (\ref{eq:W_n}) and, by Proposition~\ref{prop: invariance W}, we have that $F^n(K)\subset W_{n+j}$.
Let $P=(z_0,w_0)\in K \subset W_j$ and observe that $z_n^m \in \H^+$, so $\re (z_n^m)>0$, hence our goal is to prove that 
$$
\frac{\re (z_n^m)}{n}= \frac{\left|\re (z_n^m)\right|}{n} \longrightarrow \infty \,\,\,\,\,\,\,\,\,\text{for } n\rightarrow \infty  .
$$
Since we are interested in $|\re (z_n^m)|$, we can consider $\tilde{z}_n=z_n\,e^{\frac{2(m-j-n)}{m}\pi i}=|z_n| e^{i \tilde{\theta}_n}$  (with $-\frac{\pi}{2m}<\tilde{\theta}_n<\frac{\pi}{2m}$) instead of $z_n$, as 
$$
\left|\re (z_n^m)\right|=\left|\re(\tilde{z}_n^m)\right|=\left|z_n\right|^m \, \left|\cos (m \tilde{\theta}_n)\right| .
$$

Denote by $\alpha_n$ the angle such that $\left|\tan (\alpha_n) \right|= C \,\,\frac{j+n+1}{j+n+2}$ with $C=\tan(\frac{\pi}{2m})$,
and using (\ref{eq:W_n}), we have that $\cos(\tilde{\theta}_n)> \cos(\alpha_n) $ and  

\begin{equation}\label{eq:mtheta}
    \cos(m\tilde{\theta}_n)> \cos(m\alpha_n).
\end{equation}
It is easy to check that
$$  \left| \cos (\alpha_n) \right|=\frac{j+n+1}{\sqrt{C^2 (n+j+1)^2+(n+j+2)^2}} =:U(n)
$$
and by Remark \ref{rem:PolyC}, 
\begin{equation}\label{eq:PC}
 \left| \cos (m \alpha_n) \right|=T_m(U(n)) .
\end{equation}

Using the explicit expressions of the iterates of $F$, that is equations (\ref{eq:iteratesp}) and (\ref{eq:iteratesd}), and the fact that $\delta>2$,
we have that 
$$
\left| z_n\right|^m \ge
\begin{cases}
\delta^\frac{mn}{2} \left| z_0-1 \right|^m &\text{ if $n$ is even}\\
\delta^\frac{mn}{2} \left| w_0-1 \right|^m &\text{ if $n$ is odd} .
\end{cases} 
$$
Since $|z_0|>R_0>2$, we have that $|z_0-1|\ge |z_0|-1\ge 1$ and the same holds true for $w_0$, hence 
\begin{equation}\label{eq:z^m}
    |z_n|^m\ge \delta^\frac{mn}{2}.
\end{equation}
By (\ref{eq:mtheta}), (\ref{eq:PC}) and (\ref{eq:z^m}), we conclude
$$
\frac{\left|\re (z_n)^m\right|}{n}=\frac{|z_n|^m |\cos(m \tilde{\theta_n})|}{n}\ge \delta^{\frac{mn}{2}}\frac{T_m(U(n))}{n}\longrightarrow\infty
$$
since the dominating term is $\delta^{\frac{mn}{2}}$ and $\delta>2$.

\item 
 Let $P=(z_0,w_0)\in \Omega\setminus \AA_I$, and suppose by contradiction that there is $\epsilon>0$ and $N\in \mathbb{N}$ s.t.
$$u_n(P) < - \epsilon\,\,\,\,\,\,\,\,\,\,\,\,\forall\,\,n\ge N\,\,.$$ 
So we have that
$$-\frac{\re(z_n^m)}{n} < - \epsilon\,\,\,\,\,\,\,\,\,\,\,\,\forall\,\,n\ge N\,\,,$$
that is
$$\re(z_n^m)> \epsilon \,n\,\,\,\,\,\,\,\,\,\,\,\,\forall\,\,n\ge N\,\,.$$
Since $\epsilon>0$ we have that there exists $N'>N$ such that 
$$\re(z_n^m)> C^m \,\,\,\,\,\,\,\,\,\,\,\,\forall\,\,n\ge N'\,\,,$$
where $C$ is the constant fixed in (\ref{eq:W_I}).

Since $w_n=z_{n-1}$ and since $P\in \Omega$ for hypothesis, we have that $F^n(P)=(z_n,w_n)\in W_I$ $\forall\,\,n\ge N'$, so $P=(z_0,w_0)\in F^{-n}(W_I)\subset \AA_I$, hence the contradiction.
\end{enumerate}
\end{proof}
Now consider the following lemma, here $\D\subset \C$ is an open unit disk.
\begin{lem}[Good holomorphic disks]\label{lem:Good holo disks}
Let $P\in \Omega$, than there exists $\phi:\overline{\D}\ra \Omega$ holomorphic such that
\begin{itemize}
\item $\phi(0)=P$.
\item $D:=\phi(\D)\Subset\Omega$ and $\partial D$ is analytic.
\item The one-dimensional Lebesgue measure of $\partial \phi(\D)$ intersected with $W$ is bigger than $0$. 
\end{itemize}
\end{lem}
\begin{proof}
Since $W$ is open it is enough to take $\phi(\D)\cap W\neq\emptyset$ to get positive one- dimensional Lebesgue measure of $\partial\phi(\D)\cap W$. 
Let $P\in \Omega_{ab}$, with $a,b\in \Z_m$. Since $W_{ab}$ is not empty for all $a,b\in\Z_m$ there exists $Q\in W_{ab}$. Moreover $\Omega_{ab} $ is open and connected, so there exists a simple real analytic curve in $\Omega_{ab}$ passing through $P$ and $Q$. Complexifying this curve we get a holomorphic disk passing through $P$ that we can write as $\phi(\D)$ for some $\phi$ holomorphic defined in a neighborhood of $\D$.  Up to precomposing $\phi$ with a Moebius transformation we can assume that $P=\phi(0)$. 

\end{proof}

We recall the mean value property for harmonic functions.
\begin{rem} [Mean value property] Let $\D \subset \C$ be an open unit disk and $\phi:\overline{\D}\ra \Omega$ a  holomorphic map. Let $u$ be harmonic on $D=\phi({\D})$ and continuous up to the boundary of $D$.  Let $P_0:=\phi(0)$, then 
$$
u(P_0)=\frac1{2\pi}\int_{\partial \D}u(\zeta)|\phi'(\zeta)|^{-1}d\zeta .
$$
\end{rem}


\begin{proof}[Proof of Proposition~\ref{prop:Absorbing WI}]

Let $P\in \Omega \setminus \AA_I$ and  $D:=\phi(\D) $ as in Lemma~\ref{lem:Good holo disks}.  Let $\mu$ be the pushforward  under  $\phi$ of the one-dimensional Lebesgue measure on $\partial \D$. Let $K \subset W$ compact such that $\mu (\partial D \cap K)$ is strictly positive.  

Let $\mu_{\text{good}}=\mu(\partial D \cap K)>0$ and $\mu_{\text{bad}}=\mu(\partial D\cap (\Omega \setminus K))$. Since $D\subset \Omega$, than $\partial D=(\partial D\cap K )\cup (\partial D\cap (\Omega\setminus K))$, moreover $K$ is compact and $\Omega$ is  open, than all these sets are measurable.  

By Lemma~\ref{lem:harmonic functions} for any $\MM>0$ there exists $N$ such that $u_N\leq -\MM$ on $K$, $u_N(P)\geq -\epsilon$ for  some $\epsilon>0$  since $P\in \Omega\setminus \AA_I$, and $u_N\leq \log M$  on  $\ov{D}$ (with $M=M(\ov{D})$). Using the Mean value property we have 
\begin{align*}
-\epsilon\leq u_N(P)&=\frac{1}{2\pi}\int_{\partial D} u_N(\zeta)|\phi'(\zeta)|d\zeta=\frac{1}{2 \pi}\int_{\partial D\cap K} u_N(\zeta)|\phi'(\zeta)|d\zeta+\frac{1}{2\pi}\int_{ \partial D\cap (\Omega \setminus K) } u_N(\zeta)|\phi'(\zeta)|d\zeta\leq\\
&\leq \frac{1}{2\pi} \left(-\MM \mu_{\text{good}}+ \log M\mu_{\text{bad}}\right)\cdot \sup_{\partial\D}|\phi'|^{-1}. 
\end{align*}
Since $\MM$ is arbitrarily large, this gives  a contradiction. 
\end{proof}

As a corollary of Proposition~\ref{prop:Absorbing WI} we obtain what follows.
\begin{cor}\label{cor: w_I open}
$W_I$ is an open set.
\end{cor}
\begin{proof}
Let $P \in W_I$, our goal is to find an open neighborhood $U$ of $P$ such that $U\subset W_I$.
Since $W_I\subset \Omega \cap (I\times I)$ which is open, there exist an open neighborhood $U_P$ of $P$ compactly contained in $\Omega \cap (I\times I)$. Recall that $W_I$ is absorbing for $\Omega$, then
\begin{equation}\label{eq:W_I open}
\exists\,\,N>0\,\,\,\,\,\, \text{such that}\,\,\,\,\,\,F^n(\overline{U_P})\subset W_I\,\,\,\,\,\,\forall\,n\ge N\,\,.
\end{equation}

As usual let $P_j:=F^j(P)$ and notice that by definition of $W_I$, we have that $P_j\in W_I \subset I\times I$, which is an open set. Hence there is an open neighborhood $U_j$ of $P_j$ such that  $U_j \subset I \times I$.

Define 
$$U:=\bigcap_{j=1}^N F^{-j}(U_j) \cap U_P\,\,,$$
it is clear that $P\in U$ and $U$ is an open set since it is a finite intersection of open sets. We only need to prove that $U\subset W_I$. \\
Notice that $U\subset U_P$, hence $U$ is in the Fatou set and moreover $U\subset I \times I$. So we only need to check that $F^j(U)\subset I \times I$ for all $j \ge 0$. If $j\ge N$, this is true by (\ref{eq:W_I open}); while if $j< N$, this is true by definition since $F^j(U)\subset U_j \subset I \times I$. 
\end{proof}

\subsection{Limit sets and geometric structure of $\Omega$}\label{sect: limt sets}

We first study the  image of $(W_I)_{ab}$  under $h_1,h_2$ and then use the fact that $W_I$ is   absorbing for  $\Omega$ to understand $h_1(\Omega_{ab})$ and $h_1(\Omega_{ab})$. Moreover we show that $\Omega$ consists of $m^2$ connected components $\Omega_{ab}$, each of which is biholomorphic to $\H \times \H$.

Define the following $m$ open slices of $\C$ defined in terms of angles, all of amplitude $\frac{2 \pi}{m}$:\\

\begin{equation}\label{eq:U_j}
\begin{split}
U_0&:=\left(-\frac{\pi}{m} , \frac{\pi}{m} \right)\\
U_1&:=\left(\frac{\pi}{m} , \frac{\pi}{m}+\frac{2\pi}{m} \right)\\
U_2&:=\left(\frac{\pi}{m}+\frac{2\pi}{m} , \frac{\pi}{m}+2\frac{2\pi}{m} \right)\\
.&\\
.&\\
U_j&:=\left(\frac{\pi}{m}+(j-1)\frac{2\pi}{m} , \frac{\pi}{m}+j\frac{2\pi}{m} \right)\\
.&\\
.&\\
U_{m-1}&:=\left(\frac{\pi}{m}+(m-2)\frac{2\pi}{m} , -\frac{\pi}{m} \right)
\end{split}
\end{equation}

Observe that 
$$\C=\bigcup_{j\in \Z_m}\overline{U}_j\,.$$
Notice that if $(z,w)\in (W_I)_{ab}$, the ratio $\frac{z}{w}\in U_{a-b}$ with $a-b$ $\mod m$.
Remember that for each $(W_I)_{ab}$, we can take $(W_I)_{0\, (b-a)}$ as the representative of the cycle. With this in mind we consider the following lemma.
 
 \begin{lem}[Limit set for $W_I$]\label{lem:limit set for WI}
 Let $U_J$ defined as in (\ref{eq:U_j}), with $j\in \Z_m$. Then
 $$ 
 h_1\left((W_I)_{0b}\right)\subseteq U_{m-b}\,\,\,\,\text{and}\,\,\,\, h_2\left((W_I)_{0b}\right)\subseteq U_{b+1} \, ,\,\,\,\,\text{ if $b\neq \frac{m-1}{2}$}
 $$
and
 $$
 h_1\left((W_I)_{0b}\right),h_2\left((W_I)_{0b}\right)\subseteq U_{\frac{m+1}{2}}\, ,\,\,\,\,\text{ if $b=\frac{m-1}{2}$}.
 $$
 \end{lem} 

\begin{proof}
Remember that $h_1(z_0,w_0)=\lim_{n\ra\infty}\frac{z_{2n}}{w_{2n}}$ and $h_2(z_0,w_0)=\lim_{n\ra\infty}\frac{z_{2n+1}}{w_{2n+1}}$. 
Hence if $(z_0,w_0)\in (W_I)_{0b}$, with $b\ne \frac{m-1}{2}$, then $\frac{z_{2n}}{w_{2n}}\in U_{m-b}$ and $\frac{z_{2n+1}}{w_{2n+1}}\in U_{b+1}$. Taking the limit we get $h_1(z_0,w_0)\in \overline{U}_{m-b}$ and $h_2(z_0,w_0)\in \overline{U}_{b+1}$.

If $m-b=b+1$, that is if $b=\frac{m-1}{2}$, we have that $\frac{z_{k}}{w_{k}}\in U_{\frac{m+1}{2}}$ and taking the limit we get $h_1(z_0,w_0) , h_2(z_0,w_0)\in \overline{U}_{\frac{m+1}{2}}$.\\
Since $W_I$ is open by Corollary~\ref{cor: w_I open} its image under a holomorphic map of maximal rank is open, hence we can replace each $\overline{U_j}$ by $U_j$.


\end{proof}



To better understand Lemma~\ref{lem:limit set for WI}, let us consider two examples: $m = 5 $ and $m = 6$. In the following examples, to simplify notation, instead of writing $(W_I)_{ab}$ in the first column, we simply write $ab$.

$$\begin{displaystyle}
m=5\,\,\,\,\,\,\,\,\,\,\,\,\,\,\,\,\,\,
\begin{cases}
  (z,w) & \frac{z}{w}\,\,\,\,\,\,\textit{iteration}\\
00 & U_0 \,\,\,\,\,\,\,\,\,\,\,\, 0\\
10 & U_1 \,\,\,\,\,\,\,\,\,\,\,\, 1\\
11 & U_0 \,\,\,\,\,\,\,\,\,\,\,\, 2\\
21 & U_1 \,\,\,\,\,\,\,\,\,\,\,\, 3\\
22 & U_0 \,\,\,\,\,\,\,\,\,\,\,\, 4\\
32 & U_1 \,\,\,\,\,\,\,\,\,\,\,\, 5\\
33 & U_0 \,\,\,\,\,\,\,\,\,\,\,\, 6\\
43 & U_1 \,\,\,\,\,\,\,\,\,\,\,\, 7\\
44 & U_0 \,\,\,\,\,\,\,\,\,\,\,\, 8\\
04 & U_1 \,\,\,\,\,\,\,\,\,\,\,\, 9\\
00 & U_0 \,\,\,\,\,\,\,\,\,\,\,\, 10
\end{cases}
\,\,\,\,\,\,\,\,\,\,\,\,\,\,\,\,\,\,
\begin{cases}
(z,w) & \frac{z}{w}\,\,\,\,\,\,\textit{iteration}\\
01 & U_4 \,\,\,\,\,\,\,\,\,\,\,\, 0\\
20 & U_2 \,\,\,\,\,\,\,\,\,\,\,\, 1\\
12 & U_4 \,\,\,\,\,\,\,\,\,\,\,\, 2\\
31 & U_2 \,\,\,\,\,\,\,\,\,\,\,\, 3\\
23 & U_4 \,\,\,\,\,\,\,\,\,\,\,\, 4\\
42 & U_2 \,\,\,\,\,\,\,\,\,\,\,\, 5\\
34 & U_4 \,\,\,\,\,\,\,\,\,\,\,\, 6\\
03 & U_2 \,\,\,\,\,\,\,\,\,\,\,\, 7\\
40 & U_4 \,\,\,\,\,\,\,\,\,\,\,\, 8\\
14 & U_2 \,\,\,\,\,\,\,\,\,\,\,\, 9\\
01 & U_4 \,\,\,\,\,\,\,\,\,\,\,\, 10
  \end{cases}
  \,\,\,\,\,\,\,\,\,\,\,\,\,\,\,\,\,\,
\begin{cases}
(z,w) & \frac{z}{w}\,\,\,\,\,\,\textit{iteration}\\
02 & U_3 \,\,\,\,\,\,\,\,\,\,\,\, 0\\
30 & U_3 \,\,\,\,\,\,\,\,\,\,\,\, 1\\
13 & U_3 \,\,\,\,\,\,\,\,\,\,\,\, 2\\
41 & U_3 \,\,\,\,\,\,\,\,\,\,\,\, 3\\
24 & U_3 \,\,\,\,\,\,\,\,\,\,\,\, 4\\
02 & U_3 \,\,\,\,\,\,\,\,\,\,\,\, 5\\    \end{cases}
\end{displaystyle}$$

$$\begin{displaystyle}
m=6\,\,\,\,\,\,\,\,\,\,\,\,\,\,\,\,\,\,
  \begin{cases}
(z,w) & \frac{z}{w}\,\,\,\,\,\,\textit{iteration}\\  
00 & U_0 \,\,\,\,\,\,\,\,\,\,\,\, 0\\
10 & U_1 \,\,\,\,\,\,\,\,\,\,\,\, 1\\
11 & U_0 \,\,\,\,\,\,\,\,\,\,\,\, 2\\
21 & U_1 \,\,\,\,\,\,\,\,\,\,\,\, 3\\
22 & U_0 \,\,\,\,\,\,\,\,\,\,\,\, 4\\
32 & U_1 \,\,\,\,\,\,\,\,\,\,\,\, 5\\
33 & U_0 \,\,\,\,\,\,\,\,\,\,\,\, 6\\
43 & U_1 \,\,\,\,\,\,\,\,\,\,\,\, 7\\
44 & U_0 \,\,\,\,\,\,\,\,\,\,\,\, 8\\
54 & U_1 \,\,\,\,\,\,\,\,\,\,\,\, 9\\
55 & U_0 \,\,\,\,\,\,\,\,\,\,\,\, 10\\
05 & U_1 \,\,\,\,\,\,\,\,\,\,\,\, 11\\
00 & U_0 \,\,\,\,\,\,\,\,\,\,\,\, 12
\end{cases}
\,\,\,\,\,\,\,\,\,\,\,\,\,\,\,\,\,\,
\begin{cases}
(z,w) & \frac{z}{w}\,\,\,\,\,\,\textit{iteration}\\
01 & U_5 \,\,\,\,\,\,\,\,\,\,\,\, 0\\
20 & U_2 \,\,\,\,\,\,\,\,\,\,\,\, 1\\
12 & U_5 \,\,\,\,\,\,\,\,\,\,\,\, 2\\
31 & U_2 \,\,\,\,\,\,\,\,\,\,\,\, 3\\
23 & U_5 \,\,\,\,\,\,\,\,\,\,\,\, 4\\
42 & U_2 \,\,\,\,\,\,\,\,\,\,\,\, 5\\
34 & U_5 \,\,\,\,\,\,\,\,\,\,\,\, 6\\
53 & U_2 \,\,\,\,\,\,\,\,\,\,\,\, 7\\
45 & U_5 \,\,\,\,\,\,\,\,\,\,\,\, 8\\
04 & U_2 \,\,\,\,\,\,\,\,\,\,\,\, 9\\
50 & U_5 \,\,\,\,\,\,\,\,\,\,\,\, 10\\
15 & U_2 \,\,\,\,\,\,\,\,\,\,\,\, 11\\
01 & U_5 \,\,\,\,\,\,\,\,\,\,\,\, 12
  \end{cases}
  \,\,\,\,\,\,\,\,\,\,\,\,\,\,\,\,\,\,
\begin{cases}
(z,w) & \frac{z}{w}\,\,\,\,\,\,\textit{iteration}\\
02 & U_4 \,\,\,\,\,\,\,\,\,\,\,\, 0\\
30 & U_3 \,\,\,\,\,\,\,\,\,\,\,\, 1\\
13 & U_4 \,\,\,\,\,\,\,\,\,\,\,\, 2\\
41 & U_3 \,\,\,\,\,\,\,\,\,\,\,\, 3\\
24 & U_4 \,\,\,\,\,\,\,\,\,\,\,\, 4\\
52 & U_3 \,\,\,\,\,\,\,\,\,\,\,\, 5\\
35 & U_4 \,\,\,\,\,\,\,\,\,\,\,\, 6\\
03 & U_3 \,\,\,\,\,\,\,\,\,\,\,\, 7\\
40 & U_4 \,\,\,\,\,\,\,\,\,\,\,\, 8\\
14 & U_3 \,\,\,\,\,\,\,\,\,\,\,\, 9\\
51 & U_4 \,\,\,\,\,\,\,\,\,\,\,\, 10\\
25 & U_3 \,\,\,\,\,\,\,\,\,\,\,\, 11\\
02 & U_4 \,\,\,\,\,\,\,\,\,\,\,\, 12
  \end{cases}
\end{displaystyle}$$

\begin{figure}[h!]
\centering
\includegraphics[width=0.25\textwidth]{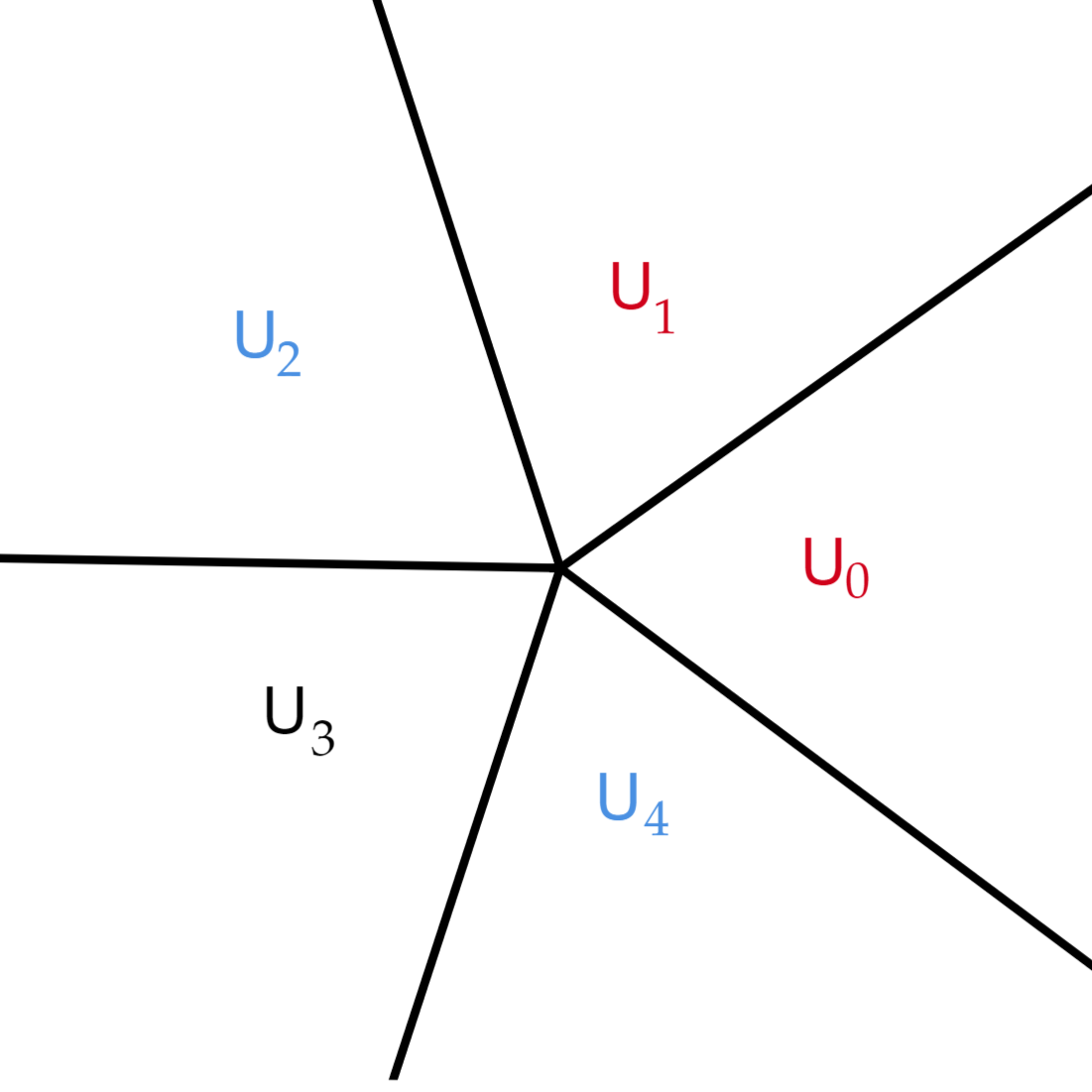} \quad \,\,\,\,\,\,\,\,\,\,\includegraphics[width=0.25\textwidth]{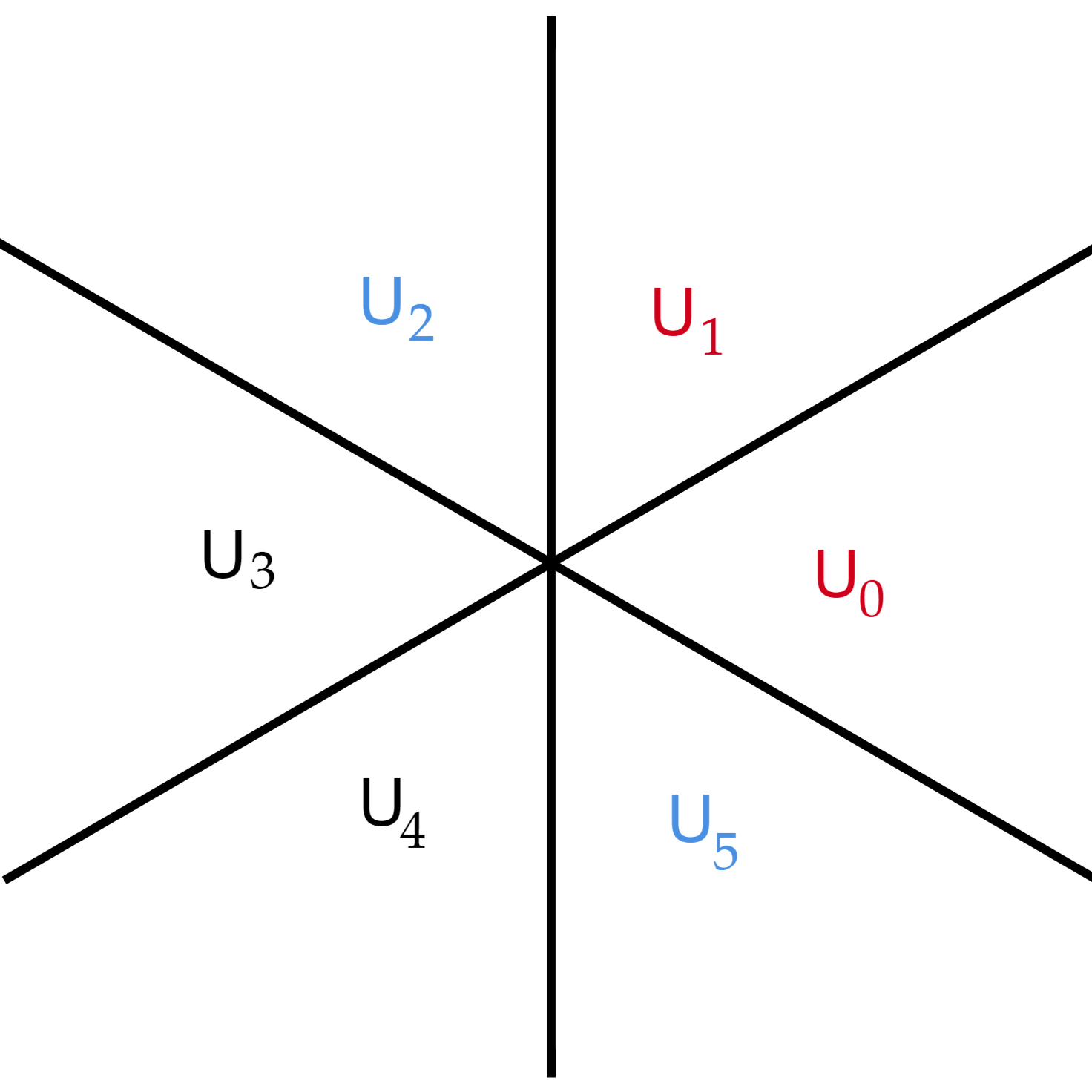}
\caption{case $m=5$ e $m=6$ on the line at infinity without the point at infinity: $\P^1\setminus \{\infty\}$ }
\label{fig:m=5,6}
\end{figure}

To better understand, see Figure 1 and observe that the components of $W_I$ belonging to the same cycle are mapped, under $h_1$ and $h_2$, into two distinct sectors $U_j$ and $U_k$, such that $j+k=1 \mod m$; with the exception of the components of $W_I$ that belong to the short cycle (this only occurs in the case of odd $m$), in which case they are mapped, under $h_1$ and $h_2$, into the same sector $U_{\frac{m+1}{2}}$.\\

Consider the following proposition in which we show the conjugacy $\phi$ between $F$
and its linear part $L$ on $\Omega$, and then in the remark we estimate the distance between
the conjugacy and the identity map.

 \begin{prop}[Conjugacy]\label{prop:conjugacy} $F$ is conjugate to the linear map $L(z,w)=(\delta e^{\frac{2 \pi}{m}i} w,z)$ on the set $\Omega$ throught a biholomorphism $\phi$. 
 \end{prop}
 \begin{proof}
Let $a=\delta e^{\frac{2 \pi}{m}i}$, it is easy to show that $L^{-n}(z,w)=(\frac{z}{a^{n/2}},\frac{w}{a^{n/2}})$ if $n$ is even and $L^{-n}(z,w)=(\frac{w}{a^{(n-1)/2}},\frac{z}{a^{(n+1)/2}})$ if $n$ is odd.
 
 Let $\phi_n:\C^2\ra\C^2$ be the automorphisms defined as
  $$
\phi_n:=L^{-n}\circ F^n. 
 $$
We first show that $F$ is conjugate to $L$ on $W_I$.\\
Our goal is to prove that  $\phi_n$ converge to a map $\phi:\C^2\ra\C^2$   uniformly on  $W_I$ so we obtain that $\phi_n$ satisfy the functional equation $\phi_{n+1}=L^{-1}\circ\phi_n\circ F$, and so the map $\phi$ is a conjugacy between $F$ and $L$.

 Using the explicit expressions for the iterates of $F$, we compute 
 \begin{align}
 \phi_{2k}(z,w)&=\left(z+\sum_{j=1}^k a^{-j}f(z_{2j-1}),w+\sum_{j=1}^k a^{-j}f(z_{2j-2})\right),\\
\phi_{2k+1}(z,w)&=\left(z+\sum_{j=1}^k a^{-j}f(z_{2j-1}),w+\sum_{j=1}^{k+1} a^{-j}f(z_{2j-2})\right),
\end{align}
 and taking the limit we obtain, using the definitions (\ref{eq:iterates1}) and (\ref{eq:iterates2}),
 $$
 \phi (z,w)=\left(z+\Delta_1(z,w),w+\Delta_2(z,w)\right),
 $$

If $P=(z,w) \in W_I$,  then $F^n(P)=(z_n,w_n)\subset I\times I\subset S$ for all $j$, so $\Delta_1(z,w)$ and $\Delta_2(z,w)$ are convergent. Hence $\phi$
 is a holomorphic map from  $W_I$ to $\phi(W_I)$. Recall that $W_I$ is open by Corollary~\ref{cor: w_I open}.
 Moreover on $W_I$, using (\ref{eq:iterates1}) and (\ref{eq:iterates2}), we get
\begin{equation}\label{eq:norma}
 \Big \|(\phi - Id) (z,w)\Big\|=\Big \|\Big(\Delta_1(z,w),\Delta_2(z,w)\Big)\Big\|<\sqrt{2}\Delta(z,w)<\sqrt{2}.
\end{equation}

 It follows that $\phi$ is open  because $W_I$ is an unbounded set, hence if $\phi$ had rank 0 or 1, $\left \|(\phi - Id)\right \|$ could not be bounded on $W_I$.  Hence the map $\phi$ is injective by Hurwitz Theorem  (see \cite{Krantz}, Exercise 3 on page 310) because the maps $\phi_n$ are injective  and their limit has rank 2. It follows that $\phi$ is a biholomorphism between $W_I$ and  $\phi(W_I)$.
  
To extend $\phi$ to all of $\Omega$ recall that $W_I$ is absorbing for $\Omega$. So if $P\in\Omega$, we have that  $F^k(P)\in W_I$ for some $k\in\N$, hence we can define   $\phi(P)= L^{-k} \circ \phi \circ F^k(P)$. Since $F$ is an automorphism, $\phi$ extends  as a biholomorphism from $\Omega$ to $\phi(\Omega)$. \\
\end{proof}  

\begin{rem}\label{rem:radice 2}
   More generally, from (\ref{eq:norma}) it follows that if $P=(z,w)\in S$ is such that $F^n(P)\in S$ for all $n\in \N$, than 
    $$
 \Big \|(\phi - Id) (z,w)\Big\|=\Big \|\Big(\Delta_1(z,w),\Delta_2(z,w)\Big)\Big\|<\sqrt{2}\Delta(z,w)<\sqrt{2}.
$$
\end{rem}

\begin{lemma}\label{lem: phi(W_I) in S}
$\phi(\Omega)\subset S$.
\end{lemma}
\begin{proof}
We first prove that $\phi(W_I)\subset S$, than, using Proposition~\ref{prop:Absorbing WI}, we extend this result to $\Omega$.

Since $W_I$ is forward invariant and contained in $S$, using Remark~\ref{rem:radice 2}, it follows that $\phi(W_I)$ is contained in a $\sqrt{2}$-neighborhood $V$ of $W_I$.

Assume by contradiction that there exists $Q\in \phi(W_I)\setminus S$. We can assume, without loss of generality, that $Q=(z_0,w_0)\in \phi((W_I)_{00})$, with $z_0=re^{i \theta}\notin \mathcal{S}_0$.  Notice that since $W_I$ is forward invariant under $F$ and $\phi$ is a conjugacy, also $\phi(W_I)$ is forward invariant under $L$, so $L^{2mn}(Q)=( a^{mn} z, a^{mn} w )=(\delta^{mn} r e^{i \theta}, \delta^{mn} w)\in \phi((W_I)_{00})$.\\

Since $\delta^{nm} r $ tends to infinity and since $\theta$ is such that $re^{i \theta}\notin\mathcal{S}_0$, the distance of $L^{2mn}(Q)$ from  the boundary of $S_{00}$ tends to infinity, hence so does the distance of  $L^{2mn}(Q)$ from $(W_I)_{00}\subset S_{00}$, contradicting $\phi(W_I)\subset V$. Hence $\phi(W_I)\subset S$. 

Since $W_I $ is absorbing for $\Omega$ under $F$, $\phi\circ F=L\circ\phi$, and $\phi(W_I)$ is completely invariant under $L$,   we have that
\begin{equation}\label{eq:image of Omega under conjugacy}
\phi(\Omega)= \phi(\bigcup_{n\geq0} F^{-n}(W_I))=\bigcup_{n\geq0} L^{-n}(\phi(W_I))\subset \phi(W_I) \subset S.
\end{equation}
\end{proof}

\begin{prop}\label{prop: four different FC}
$\Omega$ consists of $m^2$ distinct connected components.
\end{prop}

\begin{proof}
We will prove that $\phi(\Omega)$ consists of exactly $m^2$ connected components, so, using the fact that $\phi$ is a biholomorphism, the same is true for $\Omega.$

By Corollary~\ref{cor:at most}, we have that $\Omega$ consists of at most $m^2$ connected components and again since $\phi$ is a biholomorphism, the same is true for $\phi(\Omega)$. Using Lemma~\ref{lem: phi(W_I) in S}, we have that $\phi(\Omega)\subset S$ and as usual let us define $\phi(\Omega)_{ab}$ the component of $\phi(\Omega)$ contained in $S_{ab}$, with $a,b \in \Z_m$. We conclude if we prove that $\phi(\Omega)_{ab}\ne \emptyset$ for all $a,b \in \Z_m$.

Since the sets $A_{ab}$ defined in (\ref{eq:Aab}) are contained in $\Omega$ for $M$ sufficiently large, and since (using Remark~\ref{rem:radice 2}) a $\sqrt{2}$-neighborhood of $A_{ab}$ is contained in $S_{ab}$ for $M$ sufficiently large, we have that $\phi(A_{ab}) \subset \phi(\Omega)_{ab}$ for $M$ sufficiently large, hence $\phi(\Omega)_{ab}\ne \emptyset$.
\end{proof}

We now recall a simple topological fact (for a proof see Lemma 2.19 of [BBS]) that we will use in Proposition~\ref{prop:Geometric structure}. 
\begin{rem}\label{rem:topological fact}
    Let $A,B \subset \C^n$ open and $A$ is connected. If $A \cap B \ne \emptyset$ and $\partial B \cap A = \emptyset$, then $A \subseteq B$.
\end{rem}

\begin{prop}[Geometric structure of $\Omega$]\label{prop:Geometric structure}
$\Omega$ is biholomorphic to $S$.
\end{prop}

\begin{proof}
Let $W$ defined in (\ref{eq:W}). Since $W\subset \Omega$ is invariant, we have that 
$$\bigcup_{n\in \N}F^{-n}(W) \subset \Omega\,,$$
moreover, since $\phi$ is defined on $\Omega$, it is also defined on $W$, so
\begin{equation}\label{phi(Omega)}
\bigcup_{n\in \N}L^{-n}(\phi(W)) \subset \phi (\Omega)\,\,.
\end{equation}
Let $U\subset S$ such that $S$ is a $\sqrt{2}$-neighborhood of $U$.\\
Let $Q=(z,w)\in U$ and notice that $\lambda Q=(\lambda z , \lambda w) \in U$ for all $\lambda > 1$; furthermore
$$\bigcup_{\substack{n\in \N 
\\ \lambda>1}}L^{-n}(\lambda Q)= \mu Q = (\mu z , \mu w)\,\,\,\,\,\, \text{with}\,\,\,\,\mu>0\,.$$

By varying $Q\in U$, we can cover the entire set $S$:
$$S =\bigcup_{\substack{n\in \N 
\\ \lambda>1\\ Q\in U}} L^{-n}(\lambda Q) ,$$
that is
$$S=\bigcup_{n\in \N}L^{-n}(U) .$$
In view of Remark~\ref{rem:topological fact}, let $B=\phi(W)$ and $A=U^*$, where $U^*$ is $U$ with $\re z, \re w$ sufficiently large such that $\partial( \phi(W)) \cap U^*=\emptyset$.
Notice that $U^*\cap \phi(W)\ne \emptyset$, so by Remark~\ref{rem:topological fact}, we have
\begin{equation}\label{using rem}
    \phi(W)\subseteq U^* .
\end{equation}
Moreover
\begin{equation}\label{U^*}
S=\bigcup_{n\in \N}L^{-n}(U^*) .
\end{equation}

Using Remark~\ref{rem:radice 2}, equation~(\ref{U^*}), equation~(\ref{using rem}) and Lemma~\ref{lem: phi(W_I) in S}, we obtain
$$
S=\bigcup_{n\in \N}L^{-n}(U^*) \subseteq \bigcup_nL^{-n}(\phi(W))\subset \phi (\Omega) \subset S \,\,,$$
so $\phi(\Omega)=S$.
Again since $\phi$ is a biholomorphism, the claim follows.
\end{proof}

As a corollary we have what follows.
\begin{cor}
    Each Fatou component of $\Omega$ is biholomorphic to $\H\times \H$
\end{cor}
\begin{proof}
    By Proposition~\ref{prop:Geometric structure}, 
 $\Omega$ is biholomorphic to $S$ and since $S$ has $m^2$ connected components $S_{ab}$, each of which is biholomorphic to $\H \times \H$, the same is true for $\Omega$.
\end{proof}

We now study the limit set of $\Omega$.
\begin{prop}[Hyperbolic limit sets ]\label{prop:hyperbolic FC}
$h_1(\Omega_{ab})$ and $h_2(\Omega_{ab})$ are hyperbolic.
\end{prop}

\begin{proof}

By Proposition~\ref{prop:Absorbing WI}, $W_I$ is absorbing for $\Omega$ under $F$, hence by Proposition~\ref{prop: four different FC} and because of the sets $(W_I)_{ab}$ are mapping to each other, each $(W_I)_{ab}$ is absorbing for $\Omega_{ab}$ (Fatou components of $F$) under $F^{2m}$, in particular if $(W_I)_{ab}$ belongs to the short cycle, it is absorbing for $\Omega_{ab}$ under $F^m$.\\
Consequently, $\bigcup_{k\in \Z_m} (W_I)_{(a+k)\,(b+k)}$ is absorbing for $\bigcup_{k\in \Z_m} \Omega_{(a+k)\,(b+k)}$ under $F^2$.\\

Using Lemma~\ref{lem:limit set for WI}, the fact that $W_I$ is open, and considering that for each $\Omega_{ab}$ we can take $\Omega_{0\,(b-a)}$ as the representative of the cycle, we have \\ 
$$h_1(\Omega_{0b})\subset h_1(\bigcup_{k\in \Z_m} \Omega_{k\,(b+k)})=h_1( \bigcup_{k\in \Z_m} (W_I)_{k\,(b+k)}) \subseteq \begin{cases}
  U_{m-b} \,\,\,\text{if}\,\,\, b\ne \frac{m-1}{2} \\
  U_{\frac{m+1}{2}} \,\,\,\text{if}\,\,\, b= \frac{m-1}{2}
\end{cases} $$
and
$$h_2(\Omega_{0b})\subset h_2(\bigcup_{k\in \Z_m} \Omega_{k\,(b+k)})=h_2( \bigcup_{k\in \Z_m} (W_I)_{k\,(b+k)}) \subseteq \begin{cases}
  U_{b+1} \,\,\,\text{if}\,\,\, b\ne \frac{m-1}{2} \\
  U_{\frac{m+1}{2}} \,\,\,\text{if}\,\,\, b= \frac{m-1}{2}
\end{cases} $$
where $U_j$ are defined in (\ref{eq:U_j}).

So $h_1(\Omega_{ab})$ with $a,b\in \Z_m$ are hyperbolic sets.
\end{proof}

We devote the rest of this section to proving the following proposition. Again we only consider $h_i(\Omega_{0b})$ to simplify notation.
 \begin{prop}[Limit set for $\Omega$]\label{prop:limit set for Omega}
 Let $U_j$ defined in (\ref{eq:U_j}), than
$$h_1(\Omega_{0b})= \begin{cases}
  U_{m-b} \,\,\,\text{if}\,\,\, b\ne \frac{m-1}{2} \\
  U_{\frac{m+1}{2}} \,\,\,\text{if}\,\,\, b= \frac{m-1}{2}
\end{cases} $$
and
$$h_2(\Omega_{0b})= \begin{cases}
  U_{b+1} \,\,\,\text{if}\,\,\, b\ne \frac{m-1}{2} \\
  U_{\frac{m+1}{2}} \,\,\,\text{if}\,\,\, b= \frac{m-1}{2}
\end{cases} $$
 \end{prop}
To prove Proposition~\ref{prop:limit set for Omega} we shall use the following lemma.

  \begin{lem}
 \label{lem:limit set for W}
$$ 
 h_1\left((W)_{0b}\right)\supseteq U_{m-b}\,\,\,\,\text{and}\,\,\,\, h_2\left((W)_{0b}\right)\supseteq U_{b+1} \, ,\,\,\,\,\text{ if $b\neq \frac{m-1}{2}$}
 $$
and
 $$
 h_1\left((W)_{0b}\right), h_2\left((W)_{0b}\right)\supseteq U_{\frac{m+1}{2}}\, ,\,\,\,\,\text{ if $b=\frac{m-1}{2}$}.
 $$
 \end{lem}

  Before proving Lemma~\ref{lem:limit set for W} let us see how Lemma~\ref{lem:limit set for W} and Proposition~\ref{prop:hyperbolic FC} imply Proposition~\ref{prop:limit set for Omega}.
  
  \begin{proof}[Proof of Proposition~\ref{prop:limit set for Omega}]
  We prove the claim for $h_1$; for $h_2=\frac{a}{h_1}$, it follows by symmetry. 
  Since $\Omega_{ab}\supset W_{ab}$ for any $a, b \in \Z_m$, it follows that $h_1(\Omega_{ab}) \supset h_1(W_{ab})$. So in view of Lemma~\ref{lem:limit set for W}, $h_1(\Omega_{ab}) \supseteq U_j$ for some $j\in \Z_m$.
  By Proposition~\ref{prop:hyperbolic FC}, we have that $h_1(\Omega_{ab})\subseteq U_j$, and so $h_1(\Omega_{ab})=U_j$.

\end{proof}

We now give a version of Rouch\'e's Theorem in $\C^2$ (for a proof see Section 2 in \cite{BBS}). Here $\partial $ denotes the topological boundary, and  $\dist_{\operatorname{spher}}$ denotes the spherical  distance.

\begin{thm}[Rouch\'e' s Theorem in $\C^2$]\label{thm:Rouche 2D}
Let $B\subset \C^2$ be a polydisk, $F,G$ be holomorphic maps defined in a neighborhood of $\ov{B}$ which take values in $\hat{\C}$.  Let $c\in G(B)$, let $\epsilon= \dist_{\operatorname{spher}}(c, G(\partial B))>0$ and assume 
 $$
 \dist_{\operatorname{spher}}( F,G)<\epsilon \text{ on $\partial B$}.   
 $$
 Then $c\in F(B)$. 
\end{thm}
Note that $F,G$ have generic rank 1: they cannot have rank 2 because the target is $\hat{\C}$, and $G$ cannot be constant otherwise there could not be $c\in G(B)$ with positive distance from $G(\partial B)$. One can check that also $F$ cannot be constant either.

\begin{proof}[Proof of  Lemma~\ref{lem:limit set for W}]
We show that $U_0\subset h_1(W_{00})$, the other cases are analogous. Recall that orbits of points in $W$ are contained in $S$, hence Remark~\ref{rem: bounds on Delta in S} holds. Since
$$
\frac{z_{2n}}{w_{2n}}=\frac{z_0+\Delta_1^n(z_0,w_0)}{w_0+\Delta_2^n(z_0,w_0)} ,
$$
dividing the numerator and the denominator by $w_0$ and using  the fact that $\frac{1}{1+x}=\sum_{j=0}^\infty( -x)^j$ for $|x|<1$, considering $x=\frac{\Delta_2^n(z_0,w_0)}{w_0}$, we obtain
\begin{align*}
    \frac{z_{2n}}{w_{2n}}&=\left(\frac{z_0}{w_0}+\frac{\Delta_1^n(z_0,w_0)}{w_0}\right)\frac{1}{1+\frac{\Delta_2^n(z_0,w_0)}{w_0}}=\left(\frac{z_0}{w_0}+\frac{\Delta_1^n(z_0,w_0)}{w_0}\right)\sum_{j=0}^\infty \left( -\frac{\Delta_2^n(z_0,w_0)}{w_0}\right)^j=\\ &=\left(\frac{z_0}{w_0}+\frac{\Delta_1^n(z_0,w_0)}{w_0}\right)\left(1+\sum_{j=1}^\infty \left( -\frac{\Delta_2^n(z_0,w_0)}{w_0}\right)^j\right)=\\
    &=\frac{z_0}{w_0}+\frac{\Delta_1^n(z_0,w_0)}{w_0}+\left(\frac{z_0}{w_0}+\frac{\Delta_1^n(z_0,w_0)}{w_0} \right)\sum_{j=1}^\infty \left( -\frac{\Delta_2^n(z_0,w_0)}{w_0}\right)^j  .
\end{align*}

That is
\begin{equation}\label{eq: distance h0 h1}
\frac{z_{2n}}{w_{2n}}-\frac{z_{0}}{w_{0}}= \frac{\Delta_1^n(z_0,w_0)}{w_0} +\left( \frac{z_0}{w_0}+ \frac{\Delta_1^n(z_0,w_0)}{w_0}\right)\sum_{j=1}^\infty\left( \frac{-\Delta_2^n(z_0,w_0)}{w_0} \right)^{j} \text{ $\forall n\geq0$.}
\end{equation}
This expression makes sense for $|x|=\left|\frac{-\Delta_2^n(z_0,w_0)}{w_0}     \right|<1$, hence, in view of Remark~\ref{rem: bounds on Delta in S}, for  $|w_0|>1$. Recall also that $|\sum_{j=1}^\infty x^{j}|=\frac{|x|}{1-x}\leq 2|x|$ if $|x|<\frac{1}{2}$.
Let $K\subset\hat\C$  be a compact set and suppose that $\frac{z_0}{w_0}$ takes values in $K$.  By  (\ref{eq: distance h0 h1}) and using Remark~\ref{rem: bounds on Delta in S},  for   any  $\epsilon>0$    there exists $M=M(K,\epsilon)$ such that 
\begin{equation}
\bigg|\frac{z_{2n}}{w_{2n}}-\frac{z_{0}}{w_{0}}\bigg|<\epsilon \text{ for $|w_0|>M$ and $\frac{z_0}{w_0}\in K$.}
\end{equation}
Consider the function  $G(z,w):=\frac{z}{w}$
Observe that
 $$
 G^{-1}(re^{i\theta})=\{(r_1e^{i\theta_1},r_2 e^{i\theta_2})\in\C^2: \frac{r_1}{r_2}=r, \theta=\theta_1-\theta_2 \}.
 $$ 
 Let $c\in U_0$.
 By the shape of $W$ we have  that $G(W_{00})=U_0$, that  $\epsilon:=\frac{1}{2}\dist_{\operatorname{spher}}(c, G(\partial W ))>0$, and that we can choose $Q=(z_0,w_0)\in W_{00} \in G^{-1}(c)$ such that $|w_0|$ is arbitrarily large.  
By taking a limit in $n$ in equation (\ref{eq: distance h0 h1}) and  on a sufficiently small polydisk centered at $Q$ we can ensure  that  $\dist_{\operatorname{spher}}(h_1,G)<\epsilon$, hence the claim follows by Rouch\'e's Theorem.
\end{proof}

The Main Theorem \ref{Main Theorem} stated in the introduction is a direct consequence of Propositions~\ref{prop:existence of FC},
Proposition~\ref{prop:h1 h2}, Proposition~\ref{prop:conjugacy}, Proposition~\ref{prop: four different FC}, Proposition~\ref{prop:Geometric structure} and Proposition~\ref{prop:limit set for Omega}.

\bibliographystyle{amsalpha}
\bibliography{mybibliography}

\end{document}